\documentclass{amsart}
\newcommand{\Z}{\mathbb{Z}}
\newcommand{\N}{\mathbb{N}}
\newcommand{\F}{\mathbb{F}}
\newcommand{\Q}{\mathbb{Q}}
\newcommand{\C}{\mathbb{C}}
\newtheorem{Theo}{Theorem}
\newtheorem{Lem}{Lemma}

\usepackage{amssymb}
\begin{document}
\title{The non-existence of universal Carmichael numbers}
\author{Jan-Christoph Schlage-Puchta}
\begin{abstract}
We show that universal elliptic Carmichael numbers do not exist, answering a question of Silverman. Moreover, we show that the probability that an integer $n$, which is not a prime power, is an elliptic Carmichael number for a random curve $E$ with good reduction modulo $n$, is bounded above by $\mathcal{O}(\log^{-1} n)$. If we choose both $n$ and $E$ at random, the probability that $n$ is $E$-carmichael is bounded above by $\mathcal{O}(n^{-1/8+\epsilon})$.
\end{abstract}
\maketitle
\begin{center}
{\em To the memory of Wolfgang Schwarz}
\end{center}
\section{Introduction and results}
Let $E$ be an elliptic curve defined over $\Q$. Let $L(s, E)=\sum_{n\geq 1}\frac{a_n(E)}{n^s}$ be the $L$-series associated with $E$. If $p$ is prime, then $|E(\Z/p\Z)|=p-a_p(E)+1$. Hence, if $a_n(E)$ is computable, checking whether $|E(\Z/n\Z)|=n-a_n(E)+1$ is a necessary criterion for $n$ to be prime. Unfortunately the order of $E(\Z/n\Z)$ is not directly accessible, so in practice one chooses a point $P$ on the curve, and checks whether $(n-a_n(E)+1)P=0$. If this condition is satisfied for every $P\in E(\Z/n\Z)$, but $n$ is not a prime power, we say that $n$ is a Carmichael number for the curve $E$. 

Silverman\cite{Sil} defined a universal elliptic Carmichael number to be an integer $n$, which has at least two different prime factors, but $n$ is a Carmichael number for every elliptic curve $E$, which has good reduction modulo $n$, and asked whether universal elliptic Carmichael numbers exist. Here we show that such numbers do not exist. In fact, we can explicitly give parameters $a_p(E)$, which imply that $n$ is not an elliptic Carmichael number for the curve $E$. While very explicit, this proof leaves the possibility that an integer $n$ is elliptic Carmichael for most curves. Therefore we are also interested in the probability that an integer $n$ is elliptic Carmichael for a random curve $E$. We prove the following.

\begin{Theo}
\label{thm:main}
Let $n$ be an integer, which is not a prime power.
\begin{enumerate}
\item There exists a curve $E$ with good reduction modulo $n$, such that $n$ is not an elliptic Carmichael number for $E$.
\item As $n$ tends to infinity, the probability that $n$ is elliptic Carmichael for a random curve $E$ is $\mathcal{O}(\log^{-1} n)$.
\item If $n\in[x, 2x]$ is chosen at random, and $E$ is a random curve, then probability that $n$ is $E$-Carmichael is $\mathcal{O}(x^{-1/8+\epsilon})$.
\end{enumerate}
\end{Theo}

The implied constants are absolute and not too large, and come mostly from replacing terms of the form $\log\log n$ by $\log^\epsilon n$. 

Note that Luca and Shparlinski\cite{LS} considered the dual question to part 2 of Theorem~\ref{thm:main}, i.e. if $E$ is a fixed elliptic curve and $n$ is chosen at random, then the number of $E$-Carmichael integers in $[x,2x]$ is $\mathcal{O}(\frac{x}{\log\log x})$. The argument involved is quite different from our arguments, in particular, we do not have to consider to twists $L(s, E, \chi)$ of the $L$-series associated to $E$. 

Our notation follows the standard for analytic number theory. We denote by $\log_k$ the $k$-fold iterated logarithm, e.g. $\log_2 n=\log\log n$, by $\omega(n)$ the number of prime factors of $n$, $P^+(n)$ the largest prime factor of $n$. For an integer $n=\prod p_i^{e_i}$ we call $\gamma(n)=\prod p_i$ the squarefree kernel of $n$. For an integer $n$ and a prime number $p$, denote by $\nu_p(n)$ the exponent of $p$ in the prime decomposition of $n$, that is, the largest $k$ such that $p^k|n$. For a group $G$ we put $\exp(G)$, the exponent of $G$ to be the least integer $n>0$, such that $g^n=1$ for all $g\in G$. We use $\ll$ as a synonym of $=\mathcal{O}(\dots)$, an index at a Landau symbol indicates that the implied constants depend on the index.

\section{Preliminaries}
In this section we collect results on the $L$-series of an elliptic curve and on multiplicative number theory.

Our first two statements are well known, and can e.g. be found in \cite{Silbook}.

\begin{Lem}
\label{Lem:L coeffs}
The function $n\mapsto a_n$ is multiplicative, and satisfies $|a_n|\leq 2^{\omega(n)} \sqrt{n}$. For each $p$ there exists a comples number $\alpha_p$ of modulus $\sqrt{p}$, such that $a_{p^k}=\alpha^k+\overline{\alpha}^k$.
\end{Lem}

\begin{Lem}
\label{Lem:Exponent}
The group $E(\Z/p\Z)$ is a two-generated abelian group, and we have $\exp(E(\Z/p^k\Z)) = p^{k-1} \exp(E(\Z/p\Z))$.
\end{Lem}

Define the function $e:\N\rightarrow\N$ to be the multiplicative function satisfying $e(p^k)=p^{\lceil k/2\rceil}$.

\begin{Lem}
\label{Lem:exponent 2generated}
Let $G$ be a two generated abelian group. Then $e(|G|)|$ divides $\exp(G)$.
\end{Lem}
\begin{proof}
It follows from the classification of finitely generated abelian groups, that $G$ can be written as $G\cong\Z/n_1\Z\oplus\Z/n_2\Z$ with $n_1|n_2$. Clearly $\exp(G)=n_2$, and we have that if $p^k|n_1n_2$, then $p^{\lceil k/2\rceil}|n_2$. Hence our claim follows.
\end{proof}

We will repeatedly use the following alternative definition of an elliptic Carmichael number, which is \cite[Proposition~12]{Sil}.

\begin{Lem}
\label{Lem:Carmichael crit}
An integer $n$ is elliptic Carmichael for the curve $E$ if and only if $E$ has good reduction modulo $n$, and for each prime divisor $p$ of $n$ we have that $\exp(E(\Z/p^{\nu_p(n)}\Z))$ divides $a-a_n+1$.
\end{Lem}

Deuring\cite{Deuring} determined the number of curves modulo $p$ having prescribed order.

\begin{Lem}
\label{Lem:Deuring}
Let $p$ be a prime. Then the number of curves $E$ modulo $p$ with $|E(\Z/p\Z)|=p-a_p+1$ equals $H(4p-a_p^2)$, where $H(n)$ is the Kronecker class number, which can be computed as follows. Write $n=n'f^2$, where $n'$ is squarefree and $(n', f)=1$. Then
\[
H(n) = \frac{\sqrt{n}}{2\pi}L\left(1, \left(\frac{\cdot}{n'}\right)\right) \psi(f),
\]
where $\psi$ is the multiplicative function defined by
\[
\psi(p^k) = \begin{cases} \frac{p-p^{-k}}{p-1}, & \mbox{ if } \left(\frac{p}{n'}\right)=0,\\
1, & \mbox{ if } \left(\frac{p}{n'}\right)=1,\\
\frac{p+1-2p^{-k}}{p-1}, & \mbox{ if } \left(\frac{p}{n'}\right)=-1.
\end{cases}
\]
The function $\psi$ satisfies $1\leq \psi(f)\ll \log_2^2 f$.
\end{Lem}

Using this result Lenstra\cite{Lenstra} proved the following.

\begin{Lem}
\label{Lem:Lenstra}
Let $p$ be a prime, $S\subseteq[p-2\sqrt{p}, p+2\sqrt{p}]$ a set of integers. Then the probability that an elliptic curve $E$ chosen at random satisfies $|E(\Z/p\Z)|\in S$ is bounded above by $\mathcal{O}\left(\frac{|S|}{\sqrt{p}}\log p\log_2^2 p\right)$.
\end{Lem}

If $|S|$ is not too small, we can do better.

\begin{Lem}
\label{Lem:Lenstra large S}
For every fixed $c$ we have that if $p$ is a prime, and $S\subseteq[p-2\sqrt{p}, p+2\sqrt{p}]$ a set of integers satisfying $|S|>p^c$, then the probability $P$ that an elliptic curve $E$ chosen at random satisfies $|E(\Z/p\Z)|\in S$ satisfies
\[
P \ll_c \frac{|S|}{\sqrt{p}}\log_2^3 p.
\]
If $S\subseteq[p-(2-c)\sqrt{p}, p+(2-c)\sqrt{p}]$, we also have
\[
P\gg_c \frac{|S|}{\sqrt{p}\log_2^2 p}
\]
\end{Lem}

For the proof of Lemma~\ref{Lem:Lenstra large S} we need some zero density estimate. Since we are not interested in the implied constants, any result would do. We choose the following, because it is at the same time quite simple and very well known. We refer the reader to the Montgomery's book \cite{Montgomery} for more detailed information.

\begin{Lem}
\label{Lem:zero density}
For an integer $Q$, a real number $T>1$, and a real number $\sigma\geq 1/2$ define $N(\sigma, Q, T)$ to be the number of pairs $(\rho, \chi)$, where $\Re\rho\geq\sigma$, $|\Im\rho|\leq T$, $\chi$ is a primitive character to a module $q\leq Q$, and $L(\rho, \chi)=0$. Then we have $N(\sigma, Q, T)\ll (Q^2 T)^{\frac{12}{5}(1-\sigma)}\log^C (QT)$ for some constant $C$.
\end{Lem}
\begin{proof}[Proof of Lemma~$\ref{Lem:Lenstra large S}$.]
It suffices to show that for all $q\leq Q=p+2\sqrt{p}$ with at most $p^c$ exceptions we have that
\begin{equation}
\label{eq:L(1) bound}
\log_2^{-1} q \ll_c L\left(1, \left(\frac{\cdot}{q}\right)\right) \ll \log_2 q.
\end{equation}
In view of Lemma~\ref{Lem:zero density} it suffices to prove that (\ref{eq:L(1) bound}) holds true under the assumption that $L\left(s, \left(\frac{\cdot}{q}\right)\right)$ has no roots in the domain $\Re\;s>1-c/8$, $|\Im\;s|\leq Q$. But under these assumptions (\ref{eq:L(1) bound}) was essentially shown by Littlewood\cite{Little}.
\end{proof}
\begin{Lem}
\label{Lem:Siegel zero}
There exists an absolute constant $c>0$, such that for all prime numbers $p\in[x, x^2]$ with at most one exception we have that the probability that $a_p(E)=1$ holds true for a random curve $E$ is $\gg\frac{1}{\sqrt{p}\log p}$.
\end{Lem}
\begin{proof}
This follows immediately from Lemma~\ref{Lem:Deuring} and the fact that there is at most one modulus $q\in[x, x^2]$ for which a Siegel zero exist.
\end{proof}

We shall also use the following consequence of Baker's bound for linear forms in logarithms.

\begin{Lem}
\label{Lem:Baker}
Let $p$ be a prime, $E$ an elliptic curve. Then we have $a_{p^k}\neq 1$ for all $k>3\cdot 10^{20} \log p(46+\log_2 p)$.
\end{Lem}
\begin{proof}
Define $\alpha_p$ as in Lemma~\ref{Lem:L coeffs}. If $a_{p^k}=1$, then $\alpha^k=1-\overline{\alpha}^k$, thus 
\[
0<\Lambda = |k\log\alpha-k\log\overline{\alpha}-\log(-1)| \leq |\alpha|^{-k} = p^{-k/2}.
\]
Here $\alpha, \overline{\alpha}$ are algebraic numbers of degree 2 and height $\leq p$, hence from the work of Baker and W\"ustholz (confer e.g. \cite[Theorem~2.5]{BW}) we obtain
\[
\Lambda\geq \exp(-96^{10}\log^2 p\max(1, \log k)).
\]
Comparing these bounds we obtain
\[
\frac{k\log p}{2}\leq 96^{10}\log^2 p\max(1, \log k),
\]
\[
\frac{k}{2\log k}\leq 96^{10}\log p,
\]
which implies in particular $k\leq 3\cdot 10^{20} \log p(46+\log_2 p)$. Hence our claim follows.
\end{proof}

\begin{Lem}
\label{Lem:only k}
Let $p$ be a prime, $k$ an integer. Then there are $\leq k$ integers $\alpha_1, \ldots, \alpha_k$, such that for all $a$ and all elliptic curves $E$ we have that $a_{p^k}(E)=a$ implies $a_p(E)\in\{\alpha_1, \ldots, \alpha_k\}$. Similarly there are $k$ integers, such that for all $a$ and all elliptic curves $E$ we have that $a_{p^k}(E)\equiv a\pmod{p}$ implies $(a_p(E)\bmod p)\in\{\alpha_1, \ldots, \alpha_k\}$.
\end{Lem}
\begin{proof}
Define $\alpha$ as in Lemma~\ref{Lem:L coeffs}. There are two complex numbers with modulus $p^{k/2}$ and real part $a/2$. We may replace $\alpha$ by $\overline{\alpha}$, and may therefore assume that $\alpha^k$ is uniquely determined by $a$. Hence there are $k$ possible choices for $\alpha$ realizing $a$, and our first claim follows. The second claim follows similarly by considering $\F_{p^2}$ in place of $\C$.
\end{proof}

\begin{Lem}
\label{Lem:ABC}
The number of integers $n\in[x, x+\sqrt{x}]$, which satisfy $e(n)<\frac{n}{k}$ is $\mathcal{O}(\frac{\sqrt{x}}{k} + x^{1/3}\log^3 x)$.
\end{Lem}
\begin{proof}
Write $n=ab^2c^3$, where $a$ contains all prime divisors $p$ of $n$, such that $p^2\nmid n$, $c$ contains all prime divisors of $n$, which divide $n$ with an odd exponent $\geq 3$, and $b=\sqrt{n/ac^3}$. We have $e(n)=n/bc$, hence the number of integers $n\in[x, x+\sqrt{x}]$ satisfying $e(n)<\frac{n}{k}$ equals
\begin{multline*}
\sum_{bc>k}\left[\frac{x+\sqrt{x}}{b^2c^3}\right] - \left[\frac{x}{b^2c^3}\right] = \underset{b^2c^3<M}{\sum_{bc>k}}\frac{\sqrt{x}}{b^2c^3} + \mathcal{O}(M^{1/2})\\
 + \mathcal{O}(\#\{n\in[x, x+\sqrt{x}]: \exists b, c: bc>k, b^2c^3>M, b^2c^3|n\},
\end{multline*}
where $M$ is a parameter to be chosen later. The first sum is $\mathcal{O}(\frac{\sqrt{x}}{k})$. To estimate the set in the second error term consider all integers $n\in[x, x+\sqrt{x}]$, which can be written as $n=ab^2c^3$, where $b^2c^3>\sqrt{x}$, and $a\in [A, 2A]$, $b\in[B, 2B]$, $c\in[C, 2C]$. Clearly the whole range can be covered by $\mathcal{O}(\log^3 x)$ such intervals, hence, it suffices to estimate the maximum of the number of such $n$ over all $A, B, C$ with $2^5B^2C^3>\sqrt{x}$. If two of the three integers $a, b, c$ are determined, and $M>\sqrt{x}$, then there is at most one choice of the third such that $ab^2c^3\in[x, x+\sqrt{x}]$, thus we obtain that the last error term is 
\[
\ll \log^3 x \max\{\min(AB, AC, BC): AB^2C^3\leq x, B^2C^3\geq M\}.
\]
It is easy to see that the maximum is attained for $AB^2C^3=x$, hence the second factor becomes
\[
 \max\{\min(xB/M, xC/M, BC): B^2C^3\geq M\}.
\]
The function is non-decreasing in both $B$ and $C$, hence, we may assume that $B^2C^3=M$, and we obtain that the last quantity is
\[
\max\{\min(x/(BC^3), x/(B^2C^2), BC): B^2C^3=M\}\ll x^{1/3},
\]
since we can neglect the first term in the $\min$ and the side condition. Hence our claim follows.
\end{proof}

\section{Integers with special multiplicative structure in short intervals}

In this section we estimate the number of integers $n$ in an interval of the form $[x, x+c\sqrt{x}]$, which satisfy certain constraints concerning their prime factorization.

\begin{Lem}
\label{Lem:short intervals not divide n+1}
Let $\mathcal{P}$ be a set of prime numbers, and let $x$ be a sufficiently large  real number. Then the number $N$ of integers $n$ in the interval $[x, x+\sqrt{x}]$ which can be written as $dt$, where $d<x^{2/3}$ and $t$ contains only prime factors from the set $\mathcal{P}$ satisfies
\[
N \ll |\mathcal{P}|+\frac{\sqrt{x}\log|\mathcal{P}|}{\log x}.
\]
If $|\mathcal{P}|<\frac{\log x}{4\log_2 x}$, then $N\ll |\mathcal{P}|$.
\end{Lem}
\begin{proof}
Let $\mathcal{N}\subseteq[x, x+\sqrt{x}]$ be the set of integers which contain only prime factors from $\mathcal{P}$. We have
\[
\prod_{n\in\mathcal{N}} n \leq x^{2|\mathcal{N}|/3} \prod_{p\in \mathcal{P}}\prod_{n\in\mathcal{N}} p^{\nu_p(n)}.
\]
For a prime number $p\in \mathcal{P}$ the highest power of $p$ dividing an element of $[x, x+\sqrt{x}]$ is $\leq x+\sqrt{x}$. Subtracting this element from all other elements of $\mathcal{N}$ we see that 
\begin{eqnarray*}
\sum_{n\in\mathcal{N}} \nu_p(n) & \leq & \frac{\log(x+\sqrt{x})}{\log p} + \sum_{k\geq 1} k\min\left(N, \left[\frac{\sqrt{x}}{p^k}\right]\right)\\
 & \leq & \frac{\log(x+\sqrt{x})}{\log p} + 4\min\left(N, \frac{\sqrt{x}}{p}\right)
\end{eqnarray*}
On the other hand we have $\prod_{n\in\mathcal{N}} n\geq x^{N}$. If we compare these bounds we obtain
\begin{eqnarray*}
N|\log x^{1/3} & \leq & |\mathcal{P}|\log(x+\sqrt{x}) + N\underset{p<\sqrt{x}/N}{\sum_{p\in\mathcal{P}}}\log p + \sqrt{x}\underset{p>\sqrt{x}/N}{\sum_{p\in\mathcal{P}}}\frac{\log p}{p}\\
 & \leq & 2|\mathcal{P}|\log x + (1+o(1))N\log\frac{\sqrt{x}}{N}+ (1+o(1))\sqrt{x}\log|\mathcal{P}|
\end{eqnarray*}
If $N>\frac{\sqrt{x}}{\log x}$, the first claim holds true anyway. Otherwise we obtain
\[
N\leq 6|\mathcal{P}| + 2N\frac{\log_2 x}{\log x} +3\frac{\sqrt{x}\log|\mathcal{P}|}{\log x},
\]
and the first claim holds true again.
If $|\mathcal{P}|<\frac{\log_2 x}{4\log x}$, we estimate $\sum_{p\in\mathcal{P}}\log p$ using the prime number theorem to be $\leq(\frac{1}{4}+o(1))\log x$. Hence
\[
\frac{1}{3}N\log x \leq 2|\mathcal{P}|\log x+\frac{1}{4}N\log x,
\]
which implies $N\leq24|\mathcal{P}|$.
\end{proof}

Next we prove the following.

\begin{Lem}
\label{Lem:nonsmooth}
There exists some $c>0$, such that for $x$ sufficiently large there are at least $c\sqrt{x}$ integers $n\in[x, x+0.1\sqrt{x}]$ which satisfy $P^+(n)>x^{1/2+c}$.
\end{Lem}

Note that the question of finding integers with a large prime factor in a short interval has been studied since the work of Ramachandra\cite{Rama}, who proved that there is some $c>0$ such that $[x, x+x^{1/2-c}]$ contains an integer divisible by a prime factor $>x^{1/2+1/13}$, however, for the present application we need that such integers not only exist but in fact are quite frequent. Still, although we could not find this result in the literature, the methods we use here are not new, and we will therefore be quite brief. The proof relies on the following two results.

\begin{Lem}[Vaughan's identity]
\label{Lem:Vaughan}
For integers $U, V, n$ with $U<n$ we have
\[
\Lambda(n) = -\underset{d\leq V}{\underset{m\leq U}{\sum_{mdr=n}}} \Lambda(m)\mu(d) + \underset{d\leq V}{\sum_{hd=n}}\mu(d)\log h - \underset{k>1}{\underset{m>U}{\sum_{mk=n}}}\Lambda(m)\Bigg(\underset{d\leq V}{\sum_{d|k}}\mu(d)\Bigg)
\]
\end{Lem}

The following follows from Weyl's estimates.
\begin{Lem}
\label{Lem:Weyl}
There is some $\alpha<1$, such that for $N<x^{2/3}$ we have $\sum_{N\leq n<2N}e^{2\pi i x/n} \ll N^\alpha$
\end{Lem}

\begin{proof}[Proof of Lemma~\ref{Lem:nonsmooth}]
We estimate the number of integers in the interval $[x, x+0.1\sqrt{x}]$, which have a prime factor $\geq N$ by computing the sum
\begin{eqnarray*}
\sum_{n\geq N}\Lambda(n)\left(\left[\frac{x+0.1\sqrt{x}}{n}\right]-\left[\frac{x}{n}\right]\right) & = &
\sum_{n\leq x+0.1\sqrt{x}}\Lambda(n)\left[\frac{x+0.1\sqrt{x}}{n}\right] - \sum_{n\leq x+0.1\sqrt{x}}\Lambda(n)\left[\frac{x}{n}\right]\\
&&\qquad  - \sum_{n\leq N}\Lambda(n)\left(\left[\frac{x+0.1\sqrt{x}}{n}\right]-\left[\frac{x}{n}\right]\right)\\
 & = & \log[x+0.1\sqrt{x}]! - \log[x]!- \sum_{n\leq N}\Lambda(n)\left(\left[\frac{x+0.1\sqrt{x}}{n}\right]-\left[\frac{x}{n}\right]\right)\\
 & = & 0.1\sqrt{x}\log x + \mathcal{O}(\log x) - \sum_{n\leq N}\Lambda(n)\left(\left[\frac{x+0.1\sqrt{x}}{n}\right]-\left[\frac{x}{n}\right]\right).
\end{eqnarray*}
Let $B(t)=t-[t]-\frac{1}{2}$ be the saw tooth function. Then we have
\begin{eqnarray*}
\sum_{N\leq n<2N}\Lambda(n)\left(\left[\frac{x+0.1\sqrt{x}}{n}\right]-\left[\frac{x}{n}\right]\right) & = & 0.1\sqrt{x}\sum_{N\leq n<2N}\frac{\Lambda(n)}{n} + \mathcal{O}\left(\sum_{N\leq n<2N}\Lambda(n)B\left(\frac{x}{n}\right)\right)\\
 && \qquad+ \mathcal{O}\left(\sum_{N\leq n<2N}\Lambda(n)B\left(\frac{x+0.1\sqrt{x}}{n}\right)\right)\\
& = & \big(\frac{\log 2}{10}+o(1)\big)\sqrt{x} + \mathcal{O}\left(\sum_{N\leq n<2N}\Lambda(n)B\left(\frac{x}{n}\right)\right)\\
 && \qquad+ \mathcal{O}\left(\sum_{N\leq n<2N}\Lambda(n)B\left(\frac{x+0.1\sqrt{x}}{n}\right)\right)\\
\end{eqnarray*}
Approximating $B$ by a trigonometric polynomial and applying Lemma~\ref{Lem:Vaughan} and Lemma~\ref{Lem:Weyl} we see that for $N<x^{2/3}$ the error terms are $\ll N^\alpha$ for some $\alpha<1$. Putting these estimates together we obtain
\[
\sum_{n\geq N}\Lambda(n)\left(\left[\frac{x+0.1\sqrt{x}}{n}\right]-\left[\frac{x}{n}\right]\right) = 0.1\sqrt{x}\log x - 0.1\sqrt{x}\log N + \mathcal{O}(N^\alpha).
\]
Hence, if $N<c x^{1/(2\alpha)}$, and $c$ is sufficiently small, the left hand side of the last equation is  $\geq0.05\sqrt{x}\log x$, and our claim follows.
\end{proof}

\section{Proof of Theorem~\ref{thm:main}: The non-existence of universal Carmichael numbers}

In this section we prove the first part of Theorem~\ref{thm:main}. We begin with the case that $n$ is not squarefree.

\begin{Lem}
A universal elliptic Carmichael number is squarefree.
\end{Lem}
\begin{proof}
Let $q$ be a prime number such that $q^2|n$, and let $E$ be a curve  such that $a_q=0$. Then the algebraic number $\alpha$ defined in Lemma~\ref{Lem:L coeffs} equals $\sqrt{-p}$. Hence we obtain
\[
a_{q^{\nu_q(n)}}=2\Re\;(-p)^{\nu_q(n)/2)} = \begin{cases} 0, & \nu_q(n)\equiv 1\pmod{2},\\
\pm 2p^{\nu_q(n)/2}, & \nu_q(n)\equiv 0\pmod{2}.
\end{cases}
\]
By multiplicativity we conclude that $q|a_n$. Since $q|n$, we obtain that $n-a_n+1$ is not divisible by $q$. But since $q^2|n$ we have that $q|\exp(E(\Z/q^{\nu_q(n)}\Z))$, hence, from Lemma~\ref{Lem:Carmichael crit} we see that $n$ is not an elliptic Carmichael number for $E$.
\end{proof}

Next we consider the case that $n$ is squarefree.

\begin{Lem}
\label{Lem:01}
Let $n$ be a squarefree integer which has two different prime factors $p, q$, and let $E$ be an elliptic curve which satisfies $a_p(E)=1$, $a_q(E)=0$. Then $n$ is not elliptic Carmichael for the curve $E$.
\end{Lem}
\begin{proof}
From $a_q(E)=0$ and multiplicativity we obtain $a_n=0$. On the other hand $a_p=1$ implies that $\exp(E(\Z/p\Z))=p$, hence if $n$ is elliptic Carmichael for $E$, then $p$ divides both $n$ and $n-a_n+1=n+1$, which is impossible.
\end{proof}

Clearly the case $a_p=0$ is quite special. If $n$ is not the power of a single prime, we can give examples of elliptic curves for which $n$ is not $E$-Carmichael which are not supersingular for any prime divisor of $n$, however, for the prime power case we did not find such examples. Unless $n$ has a lot of very small prime divisors, the probability for the event $a_n=0$ is very small, which might leave the impression that being $E$-Carmichael is not a rare event. This impression was the main motivation for the remaining parts of Theorem~\ref{thm:main}.

\section{Proof of Theorem~\ref{thm:main}: the probability for $n$ fixed}

The proof of the second part of Theorem~\ref{thm:main} follows a bootstrap strategy. Our aim is to show that a potential counterexample has to be divisible by many small prime numbers to the first power. To do so we show that if $n$ is a counterexample, then $n$ is not divisible by large primes, and by only few medium sized primes, and that $\frac{n}{\gamma(n)}$ is small. We shall begin with rather weak results in this direction, and then use the results on large prime divisors to strengthen the results on $\gamma(n)$, and vice versa. 

\begin{Lem}
\label{Lem:gamma1}
Let $n$ be an integer satisfying $\gamma(n)\leq 2^{-\omega(n)}\sqrt{n}$. Then the probability that $n$ is an elliptic Carmichael number for a random curve $E$ is $\mathcal{O}(e^{-c\frac{\log^{1/3} n}{\log_2 n}})$ for some $c>0$.
\end{Lem}
\begin{proof}
If $E$ is a curve, for which $n$ is an elliptic Carmichael number, then we have that $\exp(E(\Z/n\Z))$ divides $n-a_n+1$. From Lemma~\ref{Lem:Exponent} we see that $\frac{n}{\gamma(n)}$ divides $\exp(E(\Z/n\Z))$, hence $\frac{n}{\gamma(n)}|n-a_n+1$. Clearly $\frac{n}{\gamma(n)}|n$, and we have $|a_n|\leq 2^{\omega(n)}\sqrt{n}$, thus our assumptions imply $a_n=1$. By multiplicativity this implies $a_{p^{\nu_p(n)}}=\pm 1$ for all prime divisors $p$ of $n$. 

From Lemma~\ref{Lem:Baker} we find that $\nu_p(n)\ll \log p\log_2 p$ for all prime divisors $p$ of $n$. From Lemma~\ref{Lem:Lenstra} and \ref{Lem:only k} we see that the probability that $a_{p^{\nu_p(n)}}=\pm 1$ is $\ll \frac{\nu_p(n)\log p\log_2^2 p}{\sqrt{p}}\ll \frac{\log^2 p\log_2^3 p}{\sqrt{p}}$. Suppose that $n$ has a prime divisor $p_0>e^{\sqrt[3]{\log n}}$. Then the probability for $a_{p_0^{\nu_{p_0}(n)}}=\pm 1$ is $\ll e^{-\sqrt[3]{\log n}/2}\log n$, which is sufficiently small. If $n$ has no prime divisor $\geq e^{\sqrt[3]{\log n}}$, then we have for each prime divisor $p$ of $n$ that
\[
p^{\nu_p(n)} \leq e^{\mathcal{O}(\log^2  p\log_2 p)}\leq e^{\mathcal{O}(\log^{2/3} n\log_2 n)} ,
\]
thus $n$ has $\gg\frac{\log^{1/3} n}{\log_2 n}$ different prime divisors. For each of them with a bounded number of exceptions we have that the probability for the event $a_{p^{\nu_p(n)}}=\pm 1$ is $\leq 1/2$, hence, the probability for the event that $a_{p^{\nu_p(n)}}=\pm 1$ for all prime divisors of $n$ is $<e^{-c\frac{\log^{1/3} n}{\log_2 n}}$, and our claim follows.
\end{proof}
\begin{Lem}
\label{Lem:gammabound2}
There exists some $c>0$, such that for every $\epsilon>0$ there exists some $n_0$ such that all integers $n>n_0$ with $\gamma(n)<\frac{n}{\log^4 n}$ we have that the probability that $n$ is an elliptic Carmichael number for random curve $E$ is $\mathcal{O}(\log^{-1} n)$.
\end{Lem}
\begin{proof}
Assume first that $n$ is divisible by a prime number $p$, such that $p^2\nmid n$.
If $n$ is an elliptic Carmichael number for $E$, then $\exp(E(\Z/n\Z))$ is divisible by $n/\gamma(n)$, hence we have $a_n\equiv 1\pmod{n/\gamma(n)}$. By multiplicativity this implies $a_p\equiv\pm 1\pmod{n/\gamma(n)}$, that is, $a_p$ takes on only $\mathcal{O}(1+\frac{\sqrt{p}}{n/\gamma(n)})$ values. We conclude that the probability that $n$ is elliptic Carmichael for a random curve is bounded above by
\[
\mathcal{O}(\frac{\log p(\log\log p)^2}{n/\gamma(n)}+\frac{\log p(\log\log p)^2}{\sqrt{p}}) = \mathcal{O}(\frac{\log^2 p}{\sqrt{p}}).
\]
Clearly the conditions $a_p\equiv \pm 1\pmod{\frac{n}{\gamma(n)}}$ are independent for different primes $p$, hence we conclude that if the product of all primes $p|n$, such that $p^2\nmid n$ supersedes $\log^2 n\log_2^4 n$, then the probability for $n$ to be $E$-Carmichael is $\ll\frac{1}{\log n}$.

Next suppose that $p\ge 5$ is a prime, such that $p^2|n$. We have $\frac{n}{\gamma(n)}|\exp(E(\Z/n\Z))$, hence, if $n$ is elliptic Carmichael for the curve $E$, then
\[
\left.\frac{n}{\gamma(n)}\right|n-a_{n/p^{\nu_p(n)}}a_{p^{\nu_p(n)}}+1.
\]
Since $\frac{n}{\gamma(n)}|n$, and $\frac{n}{\gamma(n)}\geq p^{\nu_p(n)-1}\geq p^{\nu_p(n)/2}$, we conclude that $a_{p^{\nu_p(n)}}$ is determined in $\mathcal{O}(1)$ ways. This in turn implies that $a_p$ is determined in $\mathcal{O}(\nu_p(n))$ ways, and we find that the probability that $n$ is $E$-Carmichael for a random curve is bounded above by
\[
\frac{\nu_p(n) \log p\log_2 p}{\sqrt{p}} \ll \frac{\log n \log p\log_2 p}{\sqrt{n}}.
\]
Hence if $p>\log^4 n\log^4_2 n$, our claim follows as well.

In particular we find that if $n$ has a prime divisor $p>C\log^4 n\log_2^4 n$, then $n$ is $E$-Carmichael with probability $\mathcal{O}(\log^{-1} n)$, no matter whether $p^2|n$ or not. 

Next suppose that $n$ is divisible by 3 prime numbers $p_1, p_2, p_3$, such that $p_i>\frac{1}{10}\log n$, and $2\leq \nu_{p_i}(n)\leq 4$. Put $Q=\prod_{i=1}^3 p_i^{\nu_{p_i}(n)}$. Then we have that $a_Q$ is uniquely determined modulo $\frac{Q}{\gamma(Q)}$, and since $16\sqrt{Q}\ll\frac{Q}{\gamma(Q)}$, we have that $a_Q$ is determined in $\mathcal{O}(1)$ ways. If $a_Q$ is fixed, then $a_{p_1^{\nu_{p_1}(n)}}, a_{p_2^{\nu_{p_2}(n)}}, a_{p_3^{\nu_{p_3}(n)}}$ are three integers  with product $a_Q$, thus the number of choices for this triple is bounded above by
\[
\max_{m\leq Q}\tau_3(m)\ll e^{\mathcal{O}(\frac{\log Q}{\log_2Q})} \ll e^{\mathcal{O}(\frac{\log_2 n}{\log_3 n})}.
\]
We conclude that the number of choices for the triple $(a_{p_1}, a_{p_2}, a_{p_3})$ is $\ll\log^\epsilon n$. Each fixed triple can be reached with probability $\ll\frac{\log^\epsilon n}{\sqrt{P}}$. Hence the probability that $n$ is $E$-Carmichael is bounded above by
\[
\frac{\log^\epsilon n}{P} \ll \frac{1}{\log^{3/2-\epsilon} n},
\]
and we are done.

We now bound $\gamma(n)$. We have
\begin{eqnarray*}
\gamma(n) & = & \prod_{p|n} p\\
 & = & \underset{p^2\nmid n}{\prod_{p|n}} p  \underset{2\leq \nu_p(n)\leq 4}{\prod_{p|n}} p  \underset{\nu_p(n)\geq 5}{\prod_{p|n}} p\\
 & \leq & \log^{18} n \prod_{p<\frac{1}{10}\log n} p \prod_{p^5|n} p\\
 & \ll & n^{3/10}\log^{18}n,
\end{eqnarray*}
and we are in the range already covered by Lemma~\ref{Lem:gamma1}.
\end{proof}

\begin{Lem}
\label{Lem:largep}
Suppose that $n$ is divisible by a prime number $p>n^{0.7}$. Then the probability that $n$ is elliptic Carmichael for a random elliptic curve is $\mathcal{O}(n^{-0.05})$.
\end{Lem}
\begin{proof}
Write $n=dp$. Then we have $e(p-a_p+1)|n-a_n+1$, that is, $e(p-a_p+1)|dp-a_da_p+1$. Since $e(p-a_p+1)$ divides $p-a_p+1$, we obtain
\[
e(p-a_p+1)| n-a_n+1-d(p-a_p+1) = (d-a_d)a_p -d+1\leq 2(d+2\sqrt{d})\sqrt{p}\leq \frac{3n}{\sqrt{p}} \leq 3n^{0.65}.
\]
If $n-a_n+1-d(p-a_p+1)\neq 0$, we obtain that $e(p-a_p+1)<pn^{-0.05}$, which happens with probability $\ll n^{-0.05}$ in view of Lemma~\ref{Lem:ABC}. If this quantity vanishes, then the divisibility property becomes trivial. But then we have $a_p=\frac{d-1}{d-a_d}$, which has at most one solution. Hence the probability for the event that this quantity vanishes is $\mathcal{O}(\frac{\log p\log_2 p}{\sqrt{p}})$, which is even smaller.
\end{proof}
The previous lemma will not be used directly, but together with Lemma~\ref{Lem:gammabound2} we obtain the following, which shall be used repeatedly.
\begin{Lem}
\label{Lem:2 not small p}
Suppose that $n$ has not two prime divisors $p_1, p_2$, which satisfy $p_i>0.1\log n$, and $p_i^2\nmid n$. Then the probability that $n$ is elliptic Carmichael for a random curve $E$ is $\mathcal{O}(\log^{-1} n)$.
\end{Lem}
\begin{proof}
Let $n$ be an integer. If $P^+(n)>n^{0.7}$ or $\frac{n}{\gamma(n)}>\log^4 n$, our claim follows from Lemma~\ref{Lem:largep} or Lemma~\ref{Lem:gammabound2}, respectively. Hence assume that $n$ satisfies none of these statements. We claim that for $n$ sufficiently large this already implies that $n$ has two prime divisors as described in the lemma. To see this assume the contrary. Let $p_1, p_2$ be the two largest prime divisors of $n$ such that $p_i^2\nmid n$. We want to show that these divisors exist and satisfy $p_i>0.1\log n$. Suppose the contrary. Then we have
\[
n\leq \prod_{p<0.1\log n} p \cdot P^+(n) \cdot \left(\frac{n}{\gamma(n)}\right)^2 \leq e^{(0.1+o(1))\log n} n^{0.7} \log^8 n < n^{0.9},
\]
which gives a contradiction.
\end{proof}

\begin{Lem}
\label{Lem:p1p2}
Let $n$ be an integer, $p_1<p_2$ be prime divisors of $n$ such that $p_i^2\nmid n$. Write $n=dp_1p_2$. Let $E$ an elliptic curve such that $n$ is  elliptic Carmichael for $E$. then one of the following holds true:
\begin{enumerate}
\item $a_{p_1}$ is uniquely determined by $a_d, a_{p_2};$
\item $e(p_2-a_{p_2}+1) < 4\sqrt{p_1}(p_2/p_1)^{1/6};$
\item Putting $t=(e(p_2-a_{p_2}+1), n+1)$ we have $t|a_da_{p_2}$ and $t>p_2^{1/3}$.
\end{enumerate}
\end{Lem}
\begin{proof}
If $n$ is elliptic Carmichael for $E$, then
\[
e(p_2-a_{p_2}+1)|n-a_n+1= n-a_da_{p_1}a_{p_2} + 1.
\]
Assume that $a_d$ and $a_{p_2}$ are given. Put $t=(e(p_2-a_{p_2}+1), n+1)$. Then $t|a_d a_{p_2}$, and $a_{p_1}$ is uniquely determined modulo $q=\frac{e(p_2-a_{p_2}+1)}{t}$. On the other hand we have $|a_{p_1}|\leq 2\sqrt{p}$, hence, if $q>4\sqrt{p_1}$, then we have that $a_{p_1}$ is uniquely determined in terms of $a_{d}, a_{p_2}$. If $q<4\sqrt{p_1}$, we have that $t>p_2^{1/3}$ or $e(p_2-a_{p_2}+1)<4\sqrt{p_1}(p_2/p_1)^{1/6}$. Hence in any case one of the statements (1)--(3) has to be true.
\end{proof}

\begin{Lem}
Suppose that $n$ has $L$ different prime divisors $\leq M^2$. Then $n$ is $E$-Carmichael with probability $\mathcal{O}(\log^{-1} n)$.
\end{Lem}
\begin{proof}
For $1\leq\nu\leq M$ let $a_\nu$ be the number of prime divisors of $n$ in the interval $[\nu^2, (\nu+1)^2)$. If $p$ is a prime divisor of $n$ in $[\nu^2, (\nu+1)^2)$, then by the lower bound contained in Lemma~\ref{Lem:Lenstra} we have with probability $\geq \frac{a_\nu-1}{\nu\log\nu\log_2\nu}$ that $(p-a_p+1, n)\geq\nu$.
\end{proof}

\begin{Lem}
\label{Lem:a_n=0}
The probability that $n$ is elliptic Carmichael for a random curve $E$ and at the same time $a_n(E)=0$ is $\mathcal{O}(\log^{-1}n)$.
\end{Lem}
\begin{proof}
Suppose that $n$ has two prime divisors $p_1, p_2>e^{\log^{2/3} n}$. Then we obtain $e(p_i-a_p+1)|n+1$, and by Lemma~\ref{Lem:short intervals not divide n+1} the probability for this event is $\ll\frac{\log\omega(n+1)}{\log p_i}\leq\frac{\log_2 n}{\log^{2/3} n}$. Since these two events are independent, our claim follows in this case.

The same argument applies if $n$ has $\geq \log_2 n$ prime divisors $\geq \log^C n$, where $C$ is a sufficiently large constant.

Let $c>0$ be a constant as in Lemma~\ref{Lem:nonsmooth}. Suppose that $n$ has $\geq \log_2^2 n$ prime divisors which are $\geq\log^{2-c} n$. If $q>4\sqrt{p}$ is a prime divisor of $p-a_p+1$, then $q$ does not divide any other number of the form $p-a+1$, $|a|\leq2\sqrt{p}$. By Lemma~\ref{Lem:nonsmooth} we have that the probability for the event $P^+(p-a_p+1)>p^{1/2+c}$ is bounded away from 0, together with the fact that $n$ has at most $\log n$ different prime divisors we conclude that probability for the event that $P^+(p-a_p+1)>p^{1/2+c}$ and $P^+(p-a_p+1)\nmid n+1$ is bounded away from 0. But in the latter case we either have $a_n\neq 0$ or that $n$ is not $E$-Carmichael. Since these events are independent for different $p$, we see that in this case our claim holds true as well.

Now assume that all prime divisors of $n$ are $\leq n^{0.7}$, at most one prime divisor is  $\geq e^{\log^{2/3} n}$, and at most $\log_2 n$ prime divisors of $n$ are $\geq\log^C n$. We then give a lower bound for the product $m$ of all prime divisors $p<\log^C n$ of $n$, such that $p^2\nmid n$. We have
\[
m>\frac{n}{P^+(n) e^{\log^{2/3} n\log_2 n} (\log^C n)^{\log_2^3 n}} > n^{0.2},
\]
hence, $n$ has at least $0.1\frac{\log n}{\log_2 n}$ prime divisors $p<\log^{3/2} n$, such that $p^2\nmid n$. It then follows from Lemma~\ref{Lem:Siegel zero} that the probability for the event that there exists a prime divisor $p$ of $m$ such that $a_p=1$ is
\[
\geq 1-\left(1-\frac{c}{\log^{3/4} n\log_2 n\log_3 n}\right)^{0.1\log n/\log_2 n} \geq 1-e^{-c'\frac{\log^{1/4} n}{\log_2^2 n\log_3 n}} \geq 1-\frac{1}{\log^{-1} n}, 
\]
hence, we may assume that there exists some $p$ with $a_p=1$. But then $p-a_p+1=p|n+1$, contradicting $p|n$, and the proof is complete.
\end{proof}

Therefore it suffices to consider the probability that $n$ is elliptic Carmichael and satisfies $a_n(E)\neq 0$.

\begin{Lem}
\label{Lem:many medium}
There exists a constant $C$ such that if $n$ is an integer, which is divisible by $>\log_2 n$ prime numbers $p>\log^C n$, then  the probability that $n$ is elliptic Carmichael for a random curve $E$ is $\mathcal{O}(\log^{-1} n)$.
\end{Lem}
\begin{proof}
Let $p_1, \ldots, p_k$ be the prime divisors of $n$ which satisfy $p_i\geq \log^C n$.
We put the pair $p_1, p_i$ into Lemma~\ref{Lem:p1p2}. We may assume that $C\geq 7$. The probability that (1) holds true is $\ll\frac{\log_2 p_1\log_3^2 p_1}{\sqrt{p_1}}\ll\log^{-2} n$. Using Lemma~\ref{Lem:ABC} we see that the probability that $p_i$, $i\geq 2$ satisfies (2) is $\ll \frac{p_1^{1/3}\log p_2\log_2^2 p_2}{p_2^{5/6}} + \frac{\log^4 p_2\log_2^2 p_2}{p_2^{1/3}}$, which is also $\ll\log^{-2} n$. We see that the probability that (1) or (2) holds true for at least one index $i$ is bounded above by $\frac{k}{\log^2 n}<\frac{1}{\log n}$.

Taking for $\mathcal{P}$ the set of prime divisors of $n+1$ in Lemma~\ref{Lem:short intervals not divide n+1} and using Lemma~\ref{Lem:Lenstra} we see that the probability that $p_i$, $i\geq 2$ satisfies (3) is bounded above by
\[
\frac{\log n}{\log_2 n}\frac{\log p_i\log_2 p_i}{\sqrt{p}} + \frac{\log_2 n}{\log p}.
\]
If $p>\log^C n$, the first summand is negligible as soon as $C>2$, while the second summand becomes $\leq e^{-1}$ provided that $C$ is sufficiently large. Since the third condition of Lemma~\ref{Lem:p1p2} depends only on the second prime, we see that these events are independent, and that the probability that each $p_i$  satisfies (3) is $e^{-k}\leq\frac{1}{\log n}$.
\end{proof}

\begin{Lem}
\label{Lem:some pretty large}
Suppose that $n$ has $5$ prime divisors $p_1, \ldots, p_4$, such that $p_1>\log^4 n$, and $p_2, \ldots, p_4>e^{\sqrt{\log n}}$, or $3$ prime divisors, such that $p_1>\log^4 n$, and $p_2, p_3>p^{0.01}$. Then the probability that $n$ is elliptic Carmichael for a random curve $E$ is $\ll\log^{-1} n$.
\end{Lem}
\begin{proof}
We argue as in the proof of the previous theorem. The probability that $p_1$ satsifies (1) or that $p_i$, $i\geq 2$ satisfies (2) is $\ll\log^{-1} n$. In the first case the probability that one specific $p_i$, $i\geq 2$ satisfies (3) is $\ll\frac{\log_2^4 n}{\sqrt{\log n}}$, while in the second case the probaiblity that one specific $p_i$ satisfies (3) is $\ll\frac{\log_2^4 n}{\log n}$. Since the probabilities for different $p_i$ are independent we see that in each case the probability for the event that $n$ is elliptic Carmichael is $\ll\log^{-1} n$.
\end{proof}
\begin{Lem}
There exists a $c>0$, such that the following holds true. 
Suppose that $n$ has $\geq \log_2^3 n$ prime divisors $p$ satisfying $p>\log^{2-c/2} n$. Then the probability that $n$ is elliptic Carmichael for a random curve $E$ is $\mathcal{O}(\log^{-1} n)$.
\end{Lem}
\begin{proof}
Let $c$ be as in Lemma~\ref{Lem:nonsmooth}. We may assume that at least half of the prime divisors $p$ of $n$, which satisfy $p>\log^{2-c/2} n$ also satisfies $p<\log^C n$, where $C$ is as in Lemma~\ref{Lem:many medium}, for otherwise we can apply the latter.  Cut the interval $[\log^{2-c/2} n, \log^C n]$ into $\mathcal{O}(1)$ intervals of the form $[x, x^{1+c}]$. Then one of these intervals contains $\gg\log_2^3 n$ prime divisors of $n$, let $[x, x^{1+c}]$ be the largest one of them. In particular the number of prime divisors $p$ of $n$ with $p>x^{1+c}$ is $\mathcal{O}(\log_2^3 n)$. For each prime divisor $p$ of $n$ with $x<p<x^{1+c}$ we have that $p-a_p+1$ has a prime divisor $q>p^{1/2+c}\geq x^{1/2+c}$ with probability $\gg \frac{1}{\log p\log_2^2 p}\gg \frac{1}{\log_2 n\log_3^2 n}$. The probability that this prime is also a prime divisor of $n+1$ is $\ll\log^{-c/3} n$, hence with probability $\geq 1-(1-\frac{C}{\log_2 n\log_3^2 n})^{c\log_2^3 n}\geq 1-\mathcal{O}(\log^{-1} n)$ we have that there exists a prime $q>x^{1/2+c}$, which divides some $p-a_p+1$ for some $p|n$, but $q\nmid n+1$.

But then $n$ is Carmichael with probability $\leq\frac{1}{q}\ll\frac{1}{\log n}$, and our claim follows.
\end{proof}

We can now prove part 2 of Theorem~\ref{thm:main}. It follows from Lemma~\ref{Lem:largep}, Lemma~\ref{Lem:many medium} and Lemma~\ref{Lem:some pretty large}, that either our claim is true, or $n$ has $\geq c\frac{\log n}{\log\log n}$ prime divisors in the interval $[c\log n,  \log^{2-c} n]$. In the latter case there exists some interval $[y, y+y^{1/2}]$, $c\log n<y<\log^{2-c} n$, which contains $\geq \log^{1/2+c/3} n$ prime divisors of $n$. For each of these prime divisors the probability that $p-a_p+1$ happens to be another prime divisor of $n$ is $>\log^{-1/2+c/4} n$, and for different prime divisors these probabilities are independent. We can therefore apply \v Cernov's inequality to find that with probability $>1-\log^{-1} n$ we have that there are $>\log^{c/2} n$ prime divisors $p_1, \ldots, p_k$ of $n$, such that for each $p_i$ there exists some $q_i$, such that $p_i=q_i-a_{q_i}+1$. In particular for each of these prime divisors $p_i$ we have that $p_i|n-a_n+1$, since $p_i$ also divides $n$, we obtain $a_n\equiv 1\pmod{\prod p_i}$. Pick three prime divisors $c\log n<r_1, r_2, r_3<\log^2 n$ of $n$, which are not among the $q_i$, and such that $r_i^2\nmid n$. If such prime numbers do not exist, we discard some of the $q_i$ in such a way, that the remainder still contains $\log^{c/2} n$ primes. 

If $n$ is $E$-Carmichael, then $a_{r_1r_2r_3}$ is uniquely determined modulo $\prod p_i>8\log^3 n$, hence this product is in fact uniquely determined. If $a_{r_1r_2r_3}=0$, then our claim follows from Lemma~\ref{Lem:a_n=0}. Otherwise if $a_{r_1r_2r_3}$ is given, then $(a_{r_1}, a_{r_2}, a_{r_3})$ can be chosen in
\[
\tau_3(a_{r_1r_2r_3}) \ll e^{c\frac{\log_2 n}{\log_3 n}}
\]
ways, and for each possible choice is realized with probability $\ll\prod_{i=1}^3\frac{\log r_i\log_2 r_i}{\sqrt{r_i}}$. Hence the probability that $n$ is $E$-Carmichael is bounded above by
\[
\prod_{i=1}^3\frac{\log r_i\log_2 r_i}{\sqrt{r_i}}e^{c\frac{\log_2 n}{\log_3 n}} \ll \log^{-3/2+\epsilon} n,
\]
and the proof id complete.

\section{Proof of Theorem~\ref{thm:main}: The probability for $n$ and $E$ variable}

The proof of the third part of Theorem~\ref{thm:main} is similar to the proof of the second part, but a lot easier, since we can dispose of integers with strange multiplicative behaviour immediately.

\begin{Lem}
\label{Lem:n variable p large}
Let $n\in[x, 2x]$ and $E$ be chosen at random. Then the probability that $P^+(n)>y$ and that $n$ is $E$-Charmichael is $\mathcal{O}(x^\epsilon y^{-1/2})$.
\end{Lem}
\begin{proof}
The probability that $n$ is divisible by $P^+(n)^2$ is $\mathcal{O}(y^{-1})$, hence, we may neglect this case. Put $p=P^+(n)$. Then we obtain $p-a_p+1|n-a_n+1$, substracting $a_{n/p}(p-a_p+1)$ from the right hand side we obtain
\begin{equation}
\label{eq:n variable p large}
p-a_p+1|n-p a_{n/p}-a_{n/p}+1.
\end{equation}
If the right hand side is 0, then $n+1=(p+1)a_{n/p}$. In particular $a_{n/p}\equiv 1\pmod{p}$, thus either $a_{n/p}=1$, or $p\leq 2\sqrt{n/p}$. In the first case we have $n=p$, thus $n$ is prime, and therefore not Carmichael. In the second case we obtain $n+1\leq(p+1)a_{n/p}\leq 5\sqrt{n/p}$, hence $P^+(n)\leq 25$. Since the number of integers $n\in[x, 2x]$ with $P^+(n)$ can be bounded by some power of $\log x$, we may neglect this case as well, and assume from now on that the right hand side of (\ref{eq:n variable p large}) is non-zero. 

If $n-p a_{n/p}-a_{n/p}+1$ is a non-zero integer, then it has $\ll e^{c\frac{\log}{\log\log x}}$ divisors. Hence for $p$ and $a_{n/p}$ fixed, we have that among all possible choices for $a_p$ there are only $x^\epsilon$ satisfying (\ref{eq:n variable p large}). The probability for hitting one of these choices is $\ll p^{-1/2}\log p x^\epsilon\ll y^{-1/2}x^\epsilon$, hence our claim follows. 
\end{proof}

\begin{Lem}
\label{Lem:medium divisor}
Fix a squarefree integer $d\leq x$, which is not divisible by any prime number $<\log^3 x$. Pick an integer $n\in[x,2x]$ and an elliptic curve $E$ at random. Then the probability that $d|n$, $(d, n/d)=1$, and that $n$ is $E$-Carmichael is at most $\mathcal{O}(x^{-\frac{1}{8}+\epsilon}+x^\epsilon d^{-\frac{4}{3}})$.
\end{Lem}
\begin{proof}
Fix the curve $E$ modulo $d$, and put $H=\exp(E(\Z/d\Z))$. Then we have $H|n-a_da_{n/d}+1$, thus $a_{n/d}$ is determined modulo $\frac{H}{(H, a_{n/d})}$. Since $(H, a_{n/d})|n+1$, we conclude that $a_{n/d}$ is determined modulo $\frac{H}{(H, n+1)}$. The probability for the event $d|n$ and $(n+1, H)\geq\sqrt{H}$ is $\ll\frac{1}{\sqrt{H}d}+\frac{1}{x}$, if $(n+1, H)<\sqrt{H}$, there are $\leq \sqrt{\frac{n}{dH}}$ choices for $a_{n/d}$, which could lead to charmichael numbers. 

We next show that the number of choices for $a_{p^k}$, $p^k\|\frac{n}{d}$,  which leads to a specific $a_{n/d}\neq 0$ is $\ll x^\epsilon$. There are $\ll\log x$ prime factors of $a_{n/d}$, distributing them over the possible $a_{p^k}$ can be done in at most
\[
\binom{\Omega(a_{n/d})+\omega(n/d)}{\omega(n/d)}\leq\left(\frac{C\log x + C\frac{\log x}{\log\log x}}{\frac{\log x}{\log\log x}}\right)^{C\frac{\log x}{\log\log x}}\ll \exp(C\frac{\log x\log\log\log x}{\log\log x}).
\]
different ways.

Fixing the values $a_{p^k}$, the number of choices for $a_p$ is at most $k$ by Lemma~\ref{Lem:only k}, hence the number of possible choices for $a_p$, $p|n/d$ given $a_{p^k}$ for all $p^k\|n/d$ is bounded above by $\max n_1n_2\cdots n_\ell$, where the maximum is taken over all positive integers satisfying $n_1+\dots+n_k=\Omega(n/d)\leq \frac{\log 2x}{\log 2}$, and $\ell=\omega(d)\ll\frac{\log x}{\log\log x}$. Clearly this maximum is bounded above by
\[
\left(\frac{\log x}{\ell\log 2}\right)^\ell\ll\exp(C\frac{\log x\log\log\log x}{\log\log x}).
\]
We conclude that for any $n, d$ the probability that $H|n-a_da_{n/d}+1$ subject to the condition $a_{n/d}\neq 0$ and $(H, n+1)<\sqrt{H}$ is $\ll x^\epsilon(\frac{1}{\sqrt{H}}+\sqrt{\frac{d}{n}})$.

If $a_{n/d}=0$, then $H|n+1$, which is impossible for $(H, d)>1$, and happens with probability $\leq\frac{1}{Hd}+\frac{1}{x}$ otherwise. Summarizing we obtain that the probability for the event $d|n$, $H|n-a_da_{n/d}+1$ is $\ll x^\epsilon(\frac{1}{d\sqrt{H}}+\frac{1}{x})$. If we choose $n\in[x, 2x]$ at random, the probability for the event that $d|n$,  $H|n-a_da_{n/d}+1$, and $(H, n+1)<\sqrt{H}$ is $\ll x^\epsilon(\frac{1}{d\sqrt{H}}+\frac{1}{\sqrt{dx}})$.

It remains to show that with high probability we have that $H$ is large. Write $d=p_1\cdots p_k$, and put $H_i=|E(\Z/p_1\cdots p_i\Z)|$. Call a prime $p_i$ bad, if $(H_{i-1}, |E(\Z/ p_i\Z)|)>p_i^{1/3}$. It follows from Lemma~\ref{Lem:short intervals not divide n+1} that either $H_{i-1}>x^{1/4}$, or the probability that $p_i$ is bad is bounded by $\mathcal{O}(\frac{\log x}{p_i\log\log x})=\mathcal{O}(p_i^{-2/3})$. Hence either $H>x^{1/4}$, or the probability that the product of the bad primes is $>y$ is bounded above by $x^\epsilon y^{-2/3}$. Hence with probability $\geq 1-x^\epsilon y^{-2/3}$ we obtain $H>\min(x^{1/4}, (d/y)^{2/3})$. Splitting the interval for $y$ dyadically we obtain that the probability for the event that $n$ is carmichael is bounded above by
\[
x^\epsilon\max_{1\leq y\leq d}y^{-2/3}\left(\frac{1}{d\sqrt{(d/y)^{2/3}}}+\frac{1}{dx^{1/8}}+\frac{1}{\sqrt{xd}}\right) \ll x^\epsilon\left(\frac{1}{dx^{1/8}}+\frac{1}{\sqrt{xd}}+\frac{1}{d^{4/3}}\right)
\]
\end{proof}
Denote by $\Psi(x,y)$ the number of integers $n\leq x$, such that $P^+(n)\leq
y$. The following is a consequence of Rankin's trick, see \cite[Theorem~III.5.2]{Tenenbaum}.
\begin{Lem}
\label{Lem:Psi extrem}
We have $\Psi(x,\log^3 x)=x^{2/3+o(1)}$.
\end{Lem}
We can now prove the third part of Theorem~\ref{thm:main}. Consider first the set of integers $n$ which contain a divisor $D_1\leq d\leq D_2$, such that $d$ is squarefree, coprime to $n$, and all prime divisors of $d$ are larger then $\log^3 n$. Using Lemma~\ref{Lem:medium divisor} we obtain that the probability that $n$ satisfies this condition and is $E$-carmichael for a random curve $E$ is bounded above by 
\[
x^\epsilon\sum_{D_1 \leq d\leq  D_2} (d^{-1}x^{-1/8}+d^{-1/2}x^{-1/2}+d^{-4/3})\ll x^\epsilon(x^{-1/8}+D_2^{1/2}x^{-1/2}+D_1^{-1/3}).
\]
If $n$ does not possess such a divisor, then either $P^+(n)>D_2$, or we have $\prod_{p|n}^* p^{\nu(p)} > n/D_1$, where the product is taken over all prime divisors $p$ of $n$, which are $\leq \log^3 n$ or satisfy $\nu_p(n)\geq 2$. In the first case we can apply Lemma~\ref{Lem:n variable p large} to find that the probability that $n$ is $E$-carmichael is $\ll x^{\epsilon}D_2^{-1/2}$. In the second case we can write $n=abc$, where $P^+(a)\leq\log^3 x$, $b$ is powerful, and $c<D_1$. Using Lemma~\ref{Lem:Psi extrem} we see that for given $c$  the number of possible choices for $ab$ is $\ll (x/c)^{2/3+\epsilon}$. Summing over $c$ we find that the number of possible choices for $n$ is $x^{2/3+\epsilon}D_1^{1/3}$. Hence the probability that a random $n$ satisfies this condition is $\ll (x/D_1)^{-1/3+\epsilon}$. Summing up we find that the probability that a random $n$ is $E$-carmichael for a random curve $E$ is bounded above by
\[
x^\epsilon(x^{-1/8}+D_2^{1/2}x^{-1/2}+D_1^{-1/3} + D_1^{1/3}x^{-1/3} + D_2^{-1/2}).
\]
In a wide range of parameters, e.g. for $D_1=x^{3/8}$,  $D_2=x^{3/4}$, we have that the first term dominates the other terms, hence we conclude that the probability that $n$ is $E$-carmichael for $n$ and $E$ chosen at random is $x^{-1/8+\epsilon}$.

Note that the numerical value of the exponent can probably be improved, actually, we have no idea what the real value should be. It could well be something like $-1+\epsilon$, however, this would probably be hard to prove.

Jan-Christoph Schlage-Puchta\\
Mathematisches Institut\\
Universit\"at Rostock\\
Ulmenstra\ss e 69, Haus 3\\
18057 Rostock\\
jan-christoph.schlage-puchta@uni-rostock.de

\begin{thebibliography}{99}
\bibitem{BW} A. Baker, G. W\"ustholz, {\em Logarithmic forms and diophantine geometry}, New Mathematical Monographs 9, Cambridge University Press, Cambridge, 2007.
\bibitem{Deuring} M. Deuring, Die Typen der Multiplikatorenringe elliptischer Funktionenk\"orper, {\em Abh. Math. Sem. Hansischen Univ.} {\bf 14} (1941), 197--272. 
\bibitem{Lenstra} H. W. Lenstra, Factoring integers with elliptic curves, {\em Ann. math.} {\bf 126} (1987), 649--673.
\bibitem{Little} Littlewood, On the Class-Number of the Corpus $P(\sqrt{-k})$, {\em Proc. London Math. Soc.} (1) {\bf 27} (1928),  358--372. 
\bibitem{LS} F. Luca, I. Shparlinski, On the counting function of elliptic carmichael numbers,  arXiv:1206.3476.
\bibitem{Montgomery} H. L. Montgomery, {\em Topics in multplicative number theory},
Lecture Notes in Mathematics 227, Springer, 1971
\bibitem{Rama} K. Ramachandra, A note on numbers with a large prime factor, {\em J. London Math. Soc.}  (2) {\bf 1} (1969) 303--306.
\bibitem{Silbook} J. H. Silverman, {\em The arithmetic of elliptic curves}, Graduate texts in mathematics 106, Springer-Verlag New York, 1986. 
\bibitem{Sil} J. H. Silverman, Elliptic Carmichael numbers and elliptic
  Korselt criteria, {\em Acta Arith.} 155 (2012), 233--246.
\bibitem{Tenenbaum} G. Tenenbaum, {\em Introduction to analytic and
  probabilitstic number theory}, Cambridge studies in advanced mathematics
  {\bf 46}, Cambridge, 1995.
\end{thebibliography}
\end{document}